\documentclass[draft]{amsart}
\usepackage{amsmath}
\usepackage{amssymb}
\usepackage{graphicx}
\newtheorem{theorem}{Theorem}[section]
\newtheorem{lemma}[theorem]{Lemma}
\newtheorem{proposition}[theorem]{Proposition}

\theoremstyle{definition}

\newtheorem{remark}[theorem]{Remark}

\numberwithin{equation}{section}

\begin{document}

\title[The quantization error for self-affine measures]
{Asymptotic order of the quantization errors for self-affine measures on Bedford-McMullen
carpets}

\author{Sanguo Zhu}
\address{School of Mathematics and Physics, Jiangsu University
of Technology\\ Changzhou 213001, China.}
\email{sgzhu@jsut.edu.cn}


\begin{abstract}
Let $E$ be a Bedford-McMullen carpet determined by a set of affine mappings $(f_{ij})_{(i,j)\in G}$ and $\mu$ a self-affine measure on $E$ associated with a probability vector $(p_{ij})_{(i,j)\in G}$. We prove that, for every $r\in(0,\infty)$,
the upper and lower quantization coefficient are always positive and finite in its exact quantization dimension $s_r$. As a consequence, the $k$th quantization error for $\mu$ of order $r$ is of the same order as $k^{-\frac{1}{s_r}}$. In sharp contrast to the Hausdorff measure for Bedford-McMullen carpets, our result is independent of the horizontal fibres of the carpets.
\end{abstract}

\maketitle

\section{Introduction}
Let $m,n$ be two fixed positive integers with $2\leq m\leq n$. Let $G$ be a subset of $$\big\{0,1,\ldots,n-1\big\}\times\big\{0,1,\ldots,m-1\big\}$$
with $N:=\mbox{card}\left(G\right)\geq2$. We consider a family of affine
mappings on $\mathbb{R}^{2}$:
\begin{equation}
f_{ij}:(x,y)\mapsto\big(n^{-1}x+n^{-1}i,m^{-1}y+m^{-1}j\big),\;\;(i,j)\in G.\label{fi's}
\end{equation}
By \cite{Hut:81}, there exists a unique non-empty
compact set $E$ satisfying
\[
E=\bigcup_{(i,j)\in G}f_{ij}(E).
\]
The set $E$ is the self-affine set determined by $(f_{ij})_{(i,j)\in G}$. We also call it a Bedford-McMullen carpet.
Let $(p_{ij})_{(i,j)\in G}$ be a probability vector with $p_{ij}>0$
for all $(i,j)\in G$, there exists a unique Borel probability measure
$\mu$ satisfying
\begin{equation}
\mu=\sum_{(i,j)\in G}p_{ij}\mu\circ f_{ij}^{-1}.\label{selfaffinemeas}
\end{equation}
The measure $\mu$ is referred to as the self-affine measure associated with $(p_{ij})_{(i,j)\in G}$
and $(f_{ij})_{(i,j)\in G}$. Self-affine sets and measures in the above-mentioned cases and some more general cases have
been intensively studied in the past years; one may see \cite{Bed:84,Fal:10,King:95,LG:92,Mcmullen:84,Peres:94b}
for interesting results in this direction. Write
\begin{eqnarray*}
 &&G_{x}:=  \left\{ i:(i,j)\in G \mbox{ for some } j \right\};
\;\;G_{y}:=\left\{ j:(i,j)\in G \mbox{ for some } i \right\} ,\\
 &&G_{x,j}:=\left\{ i:(i,j)\in G\right\},\;\;q_{j}:=\sum_{i\in G_{x,j}}p_{ij}\;,j\in G_y;\;\;\theta:=\frac{\log m}{\log n}.
\end{eqnarray*}
We say that $E$ has uniform horizontal fibres if ${\rm card}(G_{x,j})$ is constant for $j\in G_y$.
By Peres \cite{Peres:94b}, the Hausdorff measure of $E$ is infinite in its Hausdorff dimension if $E$ does not have uniform horizontal fibres; otherwise its Hausdorff measure is finite and positive.

In the present paper, we further study the quantization problem for self-affine measures as defined in (\ref{selfaffinemeas}). We refer to \cite{KZ:15} for some previous work of the author and Kesseb\"{o}hmer.

The quantization problem for probability measures originated in
information theory and engineering technology (cf. \cite{GN:98,Za:63}).
Mathematically, the problem consists in estimating the asymptotic error in
the approximation of a given probability measure by discrete probability
measures with finite support in terms of $L_{r}$-metrics.
We refer to Graf and Luschgy \cite{GL:00} for rigorous mathematical foundations of
quantization theory. One may see \cite{GL:04,GL:05,Kr:08,LM:02,PK:01} for more related results.

Let $\|\cdot\|$ be a norm on $\mathbb{R}^{q}$ and $d$ the metric
induced by this norm. For each $k\in\mathbb{N}$, we write $\mathcal{D}_{k}:=\{\alpha\subset\mathbb{R}^{q}:1\leq{\rm card}(\alpha)\leq k\}$.
Let $\nu$ be a Borel probability measure on $\mathbb{R}^{q}$. The
$k$th quantization error for $\nu$ of order $r\in(0,\infty)$ is defined
by
\begin{eqnarray}\label{quanerror}
e_{k,r}(\nu):=\bigg(\inf_{\alpha\in\mathcal{D}_{k}}\int d(x,\alpha)^{r}d\nu(x)\bigg)^{\frac{1}{r}}
\end{eqnarray}
By \cite{GL:00}, the $k$th quantization error equals the error when approximating $\nu$ with discrete probability measures supported on at most $k$ points.

If the infimum in (\ref{quanerror}) is attained at some $\alpha\in\mathcal{D}_k$, then we call $\alpha$ an $k$-optimal set for $\nu$ of order $r$. The collection of all $k$-optimal sets for $\nu$ of order $r$ is denoted by $C_{k,r}(\nu)$. By Theorem 4.12 of \cite{GL:00}, $C_{k,r}(\nu)$ is non-empty provided that the moment condition $\int |x|^rd\nu(x)<\infty$ is satisfied. This condition is clearly ensured if the support of the measure $\nu$ is compact. Also, under the moment condition, we have $e_{k,r}(\nu)\to 0$ as $k$ tends to infinity (see Lemma 6.1 of \cite{GL:00}).

As natural characterizations of the asymptotics for the quantization error $e_{k,r}(\nu)$ as $k$ tends to infinity, we consider the $s$-dimensional upper and
lower quantization coefficient of order $r$, which are defined below:
\begin{eqnarray*}
\underline{Q}_{r}^{s}(\nu):=\liminf_{k\to\infty}k^{\frac{r}{s}}e_{k,r}^r(P),\;\;\overline{Q}_{r}^{s}(\nu):=\limsup_{k\to\infty}k^{\frac{r}{s}}e_{k,r}^r(\nu),\;\; s\in(0,\infty).
\end{eqnarray*}
The upper and lower quantization dimension for $\nu$ of order $r$ are defined by
\begin{eqnarray}
\overline{D}_{r}(\nu):=\limsup_{k\to\infty}\frac{\log k}{-\log e_{k,r}(\nu)},\;\underline{D}_{r}(\nu):=\liminf_{k\to\infty}\frac{\log k}{-\log e_{k,r}(\nu)}.\label{quandimdef}
\end{eqnarray}
These two quantities are respectively the critical points at which the upper and lower quantization coefficient jump from infinity to zero (cf. Proposition 11.3 of \cite{GL:00} or \cite{PK:01}).
If $\overline{D}_{r}(\nu)=\underline{D}_{r}(\nu)$, the common value
is called the quantization dimension for $\nu$ of order $r$
and denoted by $D_{r}(\nu)$.

Compared with the upper and lower quantization
dimension, the upper and lower quantization coefficient
provide us with more accurate information on the asymptotic properties
of the quantization error. Accordingly, it is usually much more difficult to examine the finiteness and positivity of the upper and lower quantization coefficient.

Next, we recall our previous work on the quantization for self-affine measures in \cite{KZ:15}.
Let $s_{r}$ be the unique solution of the following equation:
\begin{eqnarray}
\bigg(\sum_{(i,j)\in G}(p_{ij}m^{-r})^{\frac{s_{r}}{s_{r}+r}}\bigg)^{\theta}\bigg(\sum_{j\in G_{y}}(q_{j}m^{-r})^{\frac{s_{r}}{s_{r}+r}}\bigg)^{1-\theta}=1.\label{maineq1}
\end{eqnarray}
In \cite{KZ:15}, Kesseb\"{o}hmer and Zhu proved that, for every $r\in(0,\infty)$, the quantization dimension for $\mu$ of order $r$ exists and equals $s_r$. Moreover, the $s_r$-dimensional upper and lower quantization coefficient are both positive and finite if one of the following conditions is fulfilled:

\begin{enumerate}

\item[\rm (a)]$\sum_{i\in G_{x,j}}(p_{ij}q_{j}^{-1})^{\frac{s_{r}}{s_{r}+r}}$
is identical for all $j\in G_{y}$;

\item[\rm (b)]$q_{j}$ is identical for
all $j\in G_{y}$.
\end{enumerate}

While the quantization dimension is determined for $\mu$ in general, the finiteness and positivity of the upper and lower quantization coefficient are examined only for some rare cases (a) and (b); in these cases we could estimate the asymptotics of the quantization error by means of another self-affine measure. One may see \cite{KZ:15} for more details.

As the upper and lower quantization coefficient indicate the convergence order of the quantization errors, they are of significant importance in quantization theory for probability measures. In view of our previous work in \cite{KZ:15}, a natural question is, what will happen if we drop the conditions in (a) and (b). With Peres'results \cite{Peres:94b} in mind, one might compare the quantization coefficient for $\mu$ with the Hausdorff measure of $E$ and conjecture that the above assumption (a) or (b) is a necessary condition for the upper and lower quantization coefficient to be both positive and finite. However, as our main result of the present paper, we will prove
\begin{theorem}
\label{mthm1} Let $\mu$ be the self-affine measure as defined in (\ref{selfaffinemeas}). Then for every $r\in(0,\infty)$ we have
$0<\underline{Q}_{r}^{s_{r}}(\mu)\leq\overline{Q}_{r}^{s_{r}}(\mu)<\infty$.
\end{theorem}

By Theorem \ref{mthm1}, one can see that the $k$th quantization error for $\mu$ of order $r$ is of the same order as $k^{-\frac{1}{s_r}}$, independently of the horizontal fibres of $E$.

The main obstacle in the way of proving Theorem \ref{mthm1} lies in the fact that, without the assumptions (a) and (b), one can hardly transfer the sums over approximate squares (cf. Section 2) of different orders to those over approximate squares of the same order. Our main idea is to associate approximate squares with subsets of the product space $G^{\mathbb{N}}\times G_y^{\mathbb{N}}$ and vice versa. This will enable us to estimate the asymptotic quantization errors for $\mu$ by means of a natural product measure on $G^{\mathbb{N}}\times G_y^{\mathbb{N}}$. We will also need to take care of the overlapping cases which are induced by such procedures.

\section{Preliminaries}
In order to avoid degenerate cases, in the following, we always assume that
\begin{equation}
2\leq m<n,\;{\rm card}\left(G_{x}\right),{\rm card}\left(G_{y}\right)\geq2.\label{hypo1}
\end{equation}
Since norms on $\mathbb{R}^q$ are pairwise equivalent, we will always work with Euclidean metrics for convenience.
For $x\in\mathbb{R}$, let $[x]$ denote the largest integer not exceeding
$x$. For every $k\in\mathbb{N}$, we set $\ell(k):=[k\theta]$ and
\begin{equation}\label{s1}
\Omega_{k}:=\left\{ \begin{array}{ll}
G_{y}^k,&\mbox{ if }\;k<\theta^{-1}\\
\Omega_{k}:=G^{\ell(k)}\times G_{y}^{k-\ell(k)},&\mbox{ if }\;k\geq\theta^{-1}
\end{array}\right.,\;k\in\mathbb{N};\;\Omega^{*}:=\bigcup_{k\geq1}\Omega_{k}.
\end{equation}
For $\sigma=\big((i_{1},j_{1}),\ldots,(i_{\ell(k)},j_{\ell(k)}),j_{\ell(k)+1},\ldots,j_{k}\big)\in\Omega^{*}$, we define
\begin{eqnarray}\label{g8}
&&|\sigma|:=k,\;\;\mu_{\sigma}:=\mu\left(F_{\sigma}\right)=\prod_{h=1}^{\ell(k)}p_{i_{h}j_{h}}\prod_{h=\ell(k)+1}^{k}q_{j_{h}},\\
&&\sigma_{a}:=\left((i_{1},j_{1}),\ldots,(i_{\ell(k)},j_{\ell(k)})\right),\;\sigma_{b}:=\left(j_{\ell(k)+1},\ldots,j_{k}\right)\nonumber.
\end{eqnarray}
We also write $\sigma=\sigma_a\ast\sigma_b$.
For $\sigma,\tau\in\Omega^{*}$, we write $\sigma\prec\tau$ if $F_{\tau}\subset F_{\sigma}$;
and write $\sigma=\tau^{\flat}$ if $\sigma\prec\tau$ and $|\tau|=|\sigma|+1$. For a word
\begin{equation}\label{g9}
\sigma=\big((i_{1},j_{1}),\ldots,(i_{\ell(k)},j_{\ell(k)}),j_{\ell(k)+1},\ldots,j_{k}\big)\in\Omega^{*},
\end{equation}
$\sigma^{\flat}$ takes the following two possible forms:
\begin{eqnarray}\label{s14}
\left\{ \begin{array}{ll}
((i_{1},j_{1}),\ldots,(i_{\ell(k)},j_{\ell(k)}),j_{\ell(k)+1},\ldots,j_{k-1}),\;{\rm if}\;\ell(k)=\ell(k-1)\\
((i_{1},j_{1}),\ldots,(i_{\ell(k)-1},j_{\ell(k)-1}),j_{\ell(k)},\ldots,j_{k-1}),\;{\rm if}\;\ell(k)=\ell(k-1)+1
\end{array}\right..
\end{eqnarray}

We say that $\sigma,\tau\in\Omega^{*}$ are incomparable if neither
$\sigma\prec\tau$ nor $\tau\prec\sigma$. A finite set $\Gamma\subset\Omega^{*}$
is called a finite antichain if any two words $\sigma,\tau\in\Gamma$
are incomparable; a finite antichain $\Gamma$ is called maximal if
$E\subset\bigcup_{\sigma\in\Gamma}F_{\sigma}$.

To each word $\sigma$ of the form (\ref{g9}),
there correspond two numbers $p,q$:
\[
p:=\sum_{h=1}^{\ell(k)}i_{h}n^{\ell(k)-h},\;\;q:=\sum_{h=1}^{k}j_{h}m^{k-h};
\]
and a unique rectangle which is called an approximate square of order $k$:
\begin{equation}\label{g3}
F_{\sigma}:=\bigg[\frac{p}{n^{\ell(k)}},\frac{p+1}{n^{\ell(k)}}\bigg]\times\bigg[\frac{q}{m^{k}},\frac{q+1}{m^{k}}\bigg].
\end{equation}
We call $\sigma$ the location code for the approximate square $F_\sigma$.
\begin{remark}{\rm (see \cite{KZ:15})
We have the following facts about approximate squares.
\begin{enumerate}

\item[\rm (f1)] Let $|A|$ denote the diameter of a set $A\subset\mathbb{R}^{2}$. One can easily see
\[
m^{-|\sigma|}\leq|F_{\sigma}|\leq m^{-|\sigma|}\sqrt{n^{2}+1}.
\]
\item[\rm (f2)] For $\sigma,\tau\in\Omega^*$, by the definition, we have, either $F_\sigma,F_\tau$ are non-overlapping, or one is a subset of the other.
\item[\rm (f3)] For $\sigma\in\Omega^*$, let $\mu_\sigma$ be as defined in (\ref{g8}). Then by (\ref{s14}), we have
\begin{equation}\label{g5}
\frac{\mu_\sigma}{\mu_{\sigma^\flat}}\leq\max\bigg\{\max_{(i,j)\in G}\frac{p_{ij}}{q_j}\max_{\hat{j}\in G_y}q_{\hat{j}},\max_{\hat{j}\in G_y}q_{\hat{j}}\bigg\}=\max_{\hat{j}\in G_y}q_{\hat{j}}.
\end{equation}
\end{enumerate}
}\end{remark}

For $r>0$ and each $k\geq1$, we define
\begin{eqnarray}\label{newanti}
&\underline{\eta}_{r}:=\min\big\{p_{ij}q_{k}m^{-r}:(i,j)\in G,k\in G_{y}\big\};\nonumber\\
&\Upsilon_{k,r}:=\big\{\sigma\in\Omega^{*}:\mu_{\sigma^{\flat}}m^{-|\sigma^{\flat}|r}\geq \underline{\eta}_{r}^k>\mu_{\sigma}m^{-|\sigma|r}\big\},\;\psi_{k,r}:={\rm card}(\Upsilon_{k,r}).
\end{eqnarray}

For two number sequences $(a_k)_{k=1}^\infty$ and $(b_k)_{k=1}^\infty$, we write $a_k\asymp b_k$ if there exists a constant $C$ independent of $k$ such that $Cb_k\leq a_k\leq C^{-1}b_k$. As the proof of Lemma 4 in \cite{KZ:15} shows, we have
\begin{eqnarray}\label{characterization}
e_{\psi_{k,r},r}^r(\mu)\asymp\sum_{\sigma\in\Upsilon_{k,r}}\mu_\sigma m^{-|\sigma|r}.
\end{eqnarray}
\begin{remark}{\rm
Let us make some remarks about $\Upsilon_{j,r}$ and the mass distribution of $\mu$.
\item[\rm (f4)]The set $\Upsilon_{k,r}$ possesses some kind of uniformity, which allows us estimate the number of points in a $\psi_{j,r}$-optimal set $\alpha$ which are lying in disjoint neighborhoods of the approximate squares $F_\sigma,\sigma\in\Upsilon_{k,r}$. We may think of (\ref{characterization}) roughly as follows. For each $\sigma\in\Upsilon_{j,r}$, $F_\sigma$ "owns" one point $a_\sigma$ of a $\psi_{j,r}$-optimal set $\alpha$ and
\[
\int_{F_\sigma}d(x,\alpha)^rd\mu(x)\asymp\mu_\sigma m^{-|\sigma|r}.
\]
We refer to \cite{KZ:15b} for some more intuitive interpretations on such estimates.

\item[\rm (f5)] The structure of the set $\Upsilon_{k,r}$ is not clear enough for us to estimate the sum on the right side of (\ref{characterization}). Let $\sigma$ be given in (\ref{g9}). Assume that $\ell(k+1)=\ell(k)+1$. For $j\neq j_{\ell(k)+1},(i,j)\in G$ and $\hat{j}\in G_y$, we write
 \[
 \hat{\sigma}=\big((i_{1},j_{1}),\ldots,(i_{\ell(k)},j_{\ell(k)}),(i,j),j_{\ell(k)+2},\ldots,j_{k},\hat{j}\big).
 \]
One can see that $F_{ \hat{\sigma}}$ is not a subset of $F_\sigma$. Roughly speaking, approximate squares do not enjoy enough "freedom" as far as sub-approximate squares are concerned.

\item[\rm (f6)]For distinct words of the form (8), the measure $\mu$ are distributed in different manners among sub-approximate squares of them, since the vectors $(p_{ij})_{i\in G_{x,j}}$ are typically not identical for $j\in G_y$.
}\end{remark}

The facts as stated in (f5) and (f6) seem to prevent us from constructing a suitable auxiliary measure via approximate squares without the assumptions (a) and (b) (see Section 1). In order to show the finiteness of the upper quantization coefficient for $\mu$, we will "embed" the sets $\Upsilon_{j,r}$ into the product space $G^{\mathbb{N}}\times G_y^{\mathbb{N}}$, and then estimate the quantization errors by using a product measure $W$ on $G^{\mathbb{N}}\times G_y^{\mathbb{N}}$ and counting all possible overlapping cases. To establish a lower bound for the lower quantization coefficient for $\mu$, we will construct a new sequence of subsets $\mathcal{L}_{j,r}(2)$ of $\Omega^*$ such that, on one hand, they can play the same role as $\Upsilon_{j,r}$, and on the other hand, they enjoy enough "freedom" so that the corresponding integrals can be well estimated by means of the above-mentioned product measure $W$.

For convenience, in the remaining part of the paper, we write
\[
\mathcal{E}_r(\sigma):=(\mu_\sigma m^{-|\sigma|r})^{\frac{s_r}{s_r+r}},\;\sigma\in\Omega^*.
\]
Note that $\psi_{j,r}\asymp\psi_{j+1,r}$ by the proof of Lemma 1 in \cite{KZ:15}. To study the finiteness and positivity of the upper and lower quantization coefficient for $\mu$, we will show that it suffices to examine the asymptotics of the sequence $(e_{\psi_{j,r},r}(\mu))_{j=1}^\infty$. By H\"{o}lder's inequality with exponent less than one, the problem further reduces to the asymptotics of the following number sequence:
\[
\sum_{\sigma\in\Upsilon_{j,r}}\mathcal{E}_r(\sigma),\;j\geq 1.
\]
For the proof of the main theorem, we will need to go back and forth between words in $\Upsilon_{j,r}$ and subsets of $G^{\mathbb{N}}\times G_y^{\mathbb{N}}$.

\section{The finiteness of the upper quantization coefficient for $\mu$}
We denote by $\vartheta$ the empty word and define
\[
G^0=G_y^0:=\{\vartheta\};\;\;G^*:=\bigcup_{k=0}^\infty G^k,\;\;G_y^*:=\bigcup_{k=0}^\infty G_y^k.
\]
Let $\sigma=\left((i_{1},j_{1}),\ldots,(i_{k},j_{k})\right)\in G^{k}$. We define
\begin{eqnarray}\label{s2}
|\sigma|=k,\;\;\sigma|_h=((i_{1},j_{1}),\ldots,(i_{h},j_{h})),\;1\leq h\leq k;\;\;\sigma^-:=\sigma|_{k-1}.
\end{eqnarray}
For $\sigma,\omega\in G^*$ with $\sigma=\omega|_{|\sigma|}$, we write $\sigma\prec\omega$. We define $\sigma|_h$ similarly for $\sigma\in G^{\mathbb{N}}$ and $h\geq 1$.
If $\omega\in G^*$ and $\sigma\in G^{\mathbb{N}}$ satisfy $\omega=\sigma_{|\omega|}$, then we also write $\omega\prec\sigma$.
For $\sigma=((i_{1},j_{1}),\ldots,(i_{k},j_{k}))$ and
$\omega=((i_{k+1},j_{k+1}),\ldots,(i_{k+h},j_{k+h}))\in G$, we write
\[
\sigma\ast\omega:=((i_{1},j_{1}),\ldots,(i_{k},j_{k}),(i_{k+1},j_{k+1}),\ldots,(i_{k+h},j_{k+h})).
\]
For $\rho,\tau\in G_y^*$, we define $\rho^-,\rho\ast\tau$ and a partial order "$\prec$" in the same manner as we did for words in $G^*$.
For $r\in(0,\infty)$, we write
\[
P_r:=\sum_{(i,j)\in G}(p_{ij}m^{-r})^{\frac{s_r}{s_r+r}},\;Q_r:=\sum_{j\in G_y}(q_jm^{-r})^{\frac{s_r}{s_r+r}}.
\]
It is noted in the proof of Lemma 5 of \cite{KZ:15} that $P_r\geq 1\geq Q_r$.

Set $\overline{q}:=\max_{j\in G_y}q_j$ and $\overline{\eta}_r:=(\overline{q}m^{-r})^{\frac{s_r}{s_r+r}}$. We define
\begin{equation}\label{s13}
H_{1,r}:=\min\{h:\overline{\eta}_r^h<\underline{\eta}_r\}.
\end{equation}
For every $k\geq 1$ and $\sigma=((i_1,j_1),\ldots,(i_k,j_k))\in G^k$ and $\tau=(j_1,\ldots,j_k)\in G_y^k$, we write
\begin{eqnarray*}
&[\sigma]=[(i_1,j_1),\ldots,(i_k,j_k)]:=\{\omega\in G^{\mathbb{N}}:\omega|_k=\sigma\};\\
&[\tau]=[j_1,\ldots,j_k]:=\{\rho\in G_y^{\mathbb{N}}:\rho|_k=\tau\}.
\end{eqnarray*}
Now for every $\sigma\in\Upsilon_{j,r}$, we associate $F_\sigma$ to a subset of $G^{\mathbb{N}}\times G_y^{\mathbb{N}}$ in the following way:
\[
\sigma=\sigma_a\ast\sigma_b\in\Upsilon_{j,r}\mapsto [\sigma_a]\times [\sigma_b]\subset G^{\mathbb{N}}\times G_y^{\mathbb{N}}.
\]

For every $(i,j)\in G$ and $j\in G_y$, we define
\[
\widetilde{p}_{ij}:=P_r^{-1}(p_{ij}m^{-r})^{\frac{s_r}{s_r+r}},\;\;\widetilde{q}_j:=Q_r^{-1}(q_jm^{-r})^{\frac{s_r}{s_r+r}}.
\]
Let $G$ and $G_y$ be endowed with discrete topology and $G^{\mathbb{N}},G_y^{\mathbb{N}}$ be endowed with the corresponding product topology. We denote by $\mathcal{B}_1,\mathcal{B}_2$ the Borel sigma-algebra on $G^{\mathbb{N}},G_y^{\mathbb{N}}$.
By Kolmogrov consistency theorem, there exist a unique Borel probability measure $\lambda$ on $G^{\mathbb{N}}$ and a unique Borel probability measure $\nu$ on $G_y^{\mathbb{N}}$ such that
\begin{eqnarray*}
&\lambda([\sigma])=\prod_{h=1}^k\widetilde{p}_{i_hj_h},\;{\rm for\;every}\;\;\sigma=(i_1,j_1),\ldots,(i_k,j_k)\in G^k\;\;{\rm and}\;\;k\geq 1;
\\&\nu([\tau])=\prod_{h=1}^k\widetilde{q}_{i_hj_h},\;{\rm for\;every}\;\;\tau=(j_1,\ldots,j_k)\in G_y^k\;\;{\rm and}\;\;k\geq 1.
\end{eqnarray*}
Thus, we obtain a unique product measure $W$ on $G^{\mathbb{N}}\times G_y^{\mathbb{N}}$ such that
\begin{eqnarray*}
W(A\times B)=\lambda(A)\nu(B),\;\;A\in\mathcal{B}_1,\; B\in\mathcal{B}_2.
\end{eqnarray*}
We know that words in $\Upsilon_{j,r}$ are pairwise incomparable and $F_\sigma,\sigma\in\Upsilon_{j,r}$, are non-overlapping. However, it can happen that $[\sigma^{(1)}_a]\times[\sigma^{(1)}_b]$ and $[\sigma^{(2)}_a]\times[\sigma^{(2)}_b]$ are overlapping. We will use the following lemma to treat such overlapping cases.

\begin{lemma}\label{pre0}
For every $\sigma\in\Upsilon_{j,r}$, we write
\[
S_1(\sigma):=\{\tau\in\Upsilon_{j,r}:\sigma_a\prec\tau_a,\sigma_b\prec\tau_b\}.
\]
Then we have
\[
\sum_{\tau\in S_1(\sigma)}W([\tau_a]\times[\tau_b])\leq H_{1,r}W([\sigma_a]\times[\sigma_b]).
\]
\end{lemma}
\begin{proof}
For every $h\geq 1$, let $\Gamma_h([\sigma_a]\times[\sigma_b])$ denote the collection of the subsets $[\rho]\times [\omega]$ of $G^{\mathbb{N}}\times G_y^{\mathbb{N}}$ satisfying
\begin{eqnarray*}
[\rho]\times [\omega]\subset[\sigma_a]\times[\sigma_b],\;|\rho|+|\omega|=|\sigma|+h,\;[(|\rho|+\omega)\theta]=|\rho|.
\end{eqnarray*}
Note that the words in $\Gamma_1([\sigma_a]\times[\sigma_b])$ take exactly one of the following two forms:
\begin{equation}\label{z2}
[\sigma_a\ast(i,j)]\times[\sigma_b],\;\;{\rm or}\;\;[\sigma_a]\times[\sigma_b\ast \hat{j}],\;(i,j)\in G,\;\hat{j}\in G_y.
\end{equation}
Using this fact and mathematical induction, for every $h\geq 1$, we obtain
\begin{equation}\label{z10}
\sum_{\rho\times\omega\in\Gamma_h([\sigma_a]\times[\sigma_b])}W(\Gamma_h([\sigma_a]\times[\sigma_b]))=W([\sigma_a]\times[\sigma_b]).
\end{equation}
Also, using (\ref{z2}) and mathematical induction, for every $\rho\times\omega\in\Gamma_h([\sigma_a]\times[\sigma_b])$, we have
\begin{equation}\label{z3}
\underline{\eta}_r^{\frac{hs_r}{s_r+r}}\mathcal{E}_r(\sigma)\leq\mathcal{E}_r(\rho\ast\omega)\leq\overline{\eta}_r^{\frac{hs_r}{s_r+r}}\mathcal{E}_r(\sigma).
\end{equation}
By the definition, one can see that for every $\tau\in S_1(\sigma)$, we have
\begin{eqnarray*}
[\tau_a]\times[\tau_b]\in\Gamma_h([\sigma_a]\times[\sigma_b])\;\;{\rm  for\; some}\;\; h.
\end{eqnarray*}
Suppose that for some $\tau\in S_1(\sigma)$, we have $|\tau|\geq |\sigma|+H_{1,r}$. By (\ref{z3}), we would have
\[
\mathcal{E}_r(\tau)\leq \overline{\eta}_r^{H_{1,r}}\mathcal{E}_r(\sigma)<\underline{\eta}_r^{\frac{s_r}{s_r+r}}\mathcal{E}_r(\sigma).
\]
This contradicts (\ref{newanti}), since by (\ref{newanti}), for every $\tau\in\Upsilon_{j,r}$, we have
\[
\underline{\eta}_r^{\frac{s_r}{s_r+r}}\mathcal{E}_r(\sigma)\leq\mathcal{E}_r(\tau)\leq\underline{\eta}_r^{\frac{-s_r}{s_r+r}}\mathcal{E}_r(\sigma).
\]
Thus, for every $\tau\in S_1(\sigma)$, we have $|\tau|\leq|\sigma|+H_{1,r}$. It follows that
\begin{equation}\label{g1}
\bigcup_{\tau\in S_1(\sigma)}[\tau_a]\times[\tau_b]\subset\bigcup_{h=1}^{H_{1,r}}\Gamma_h([\sigma_a]\times[\sigma_b]),
\end{equation}
For distinct words $\sigma^{(1)},\sigma^{(2)}\in\Upsilon_{j,r}$, we have either $\sigma^{(1)}_a\neq\sigma^{(2)}_a$, or $\sigma^{(1)}_b\neq\sigma^{(2)}_b$. So,
\begin{equation}\label{g2}
[\sigma^{(1)}_a]\times[\sigma^{(1)}_b]\neq[\sigma^{(2)}_a]\times[\sigma^{(2)}_b].
\end{equation}
Thus, the lemma follows by (\ref{z10}), (\ref{g1}) and (\ref{g2}).
\end{proof}

Next, we show the finiteness of the upper quantization coefficient for $\mu$, by using Lemma \ref{pre0} and the auxiliary measure $W$.
\begin{proposition}\label{pre1}
Let $\mu$ be a measure as defined in (\ref{selfaffinemeas}). Then $\overline{Q}_r^{s_r}(\mu)<\infty$.
\end{proposition}
\begin{proof}
First, we estimate $\sum_{\sigma\in\Upsilon_{j,r}}\mathcal{E}_r(\sigma)$ from above by means of the measure $W$. For a word $\sigma\in\Upsilon_{j,r}$, by the definition, it takes the form:
\[
\sigma=\big((i_{1},j_{1}),\ldots,(i_{\ell(k)},j_{\ell(k)}),j_{\ell(k)+1},\ldots,j_{k}\big)\in\Omega^{*}.
\]
We associate $\sigma$ with the following subset of $G^{\mathbb{N}}\times G_y^{\mathbb{N}}$:
\[
[\sigma_a]\times[\sigma_b]=[(i_1,j_1),\ldots,(i_{\ell(k)},j_{\ell(k)})]\times[j_{\ell(k)+1},\ldots,j_{k}].
\]
Note that for all $k\geq\theta^{-1}$, we have $P_r^{-1}Q_r\leq P_r^{\ell(k)}Q_r^{(k-\ell(k))}\leq 1$. We deduce
\begin{eqnarray}\label{s3}
W([\sigma_a]\times[\sigma_b])&&=\prod_{h=1}^{\ell(k)}\widetilde{p}_{i_hj_h}\prod_{h=\ell(k)+1}^k\widetilde{q}_{j_h}\nonumber
\\&&=P_r^{-\ell(k)}Q_r^{-(k-\ell(k))}
(\mu_\sigma m^{-|\sigma|r})^{\frac{s_r}{s_r+r}}\nonumber\\&&\left\{ \begin{array}{ll}
\leq P_r Q_r^{-1}\mathcal{E}_r(\sigma)\\
\geq\mathcal{E}_r(\sigma)
\end{array}\right..
\end{eqnarray}
For distinct words $\sigma^{(1)},\sigma^{(2)}\in\Upsilon_{j,r}$, we have either $\sigma^{(1)}_a\neq\sigma^{(2)}_a$ or $\sigma^{(1)}_b\neq\sigma^{(2)}_b$. Thus, they are associated to distinct subsets of $G^{\mathbb{N}}\times G_y^{\mathbb{N}}$. We write
\[
\mathcal{W}_{j,r}:=\{[\sigma_a]\times[\sigma_b]:\;\sigma=\sigma_a\ast\sigma_b\in\Upsilon_{j,r}\}.
\]
 We distinguish two cases:

Case (i): either $\sigma^{(1)}_a,\sigma^{(2)}_a$ or, $\sigma^{(1)}_b,\sigma^{(2)}_b$ are incomparable. In this case, we have
\[
([\sigma^{(1)}_a]\times[\sigma^{(1)}_b])\cap([\sigma^{(2)}_a]\times[\sigma^{(2)}_b])=\emptyset.
\]

Case (ii): both $\sigma^{(1)}_a,\sigma^{(2)}_a$  and $\sigma^{(1)}_b,\sigma^{(2)}_b$ are comparable. Note that
\[
[(|\sigma_a|+|\sigma_b|)\theta]=|\sigma_a|,\;{\rm for\;all}\;\;\sigma=\sigma_a\ast\sigma_b\in\Upsilon_{j,r}.
\]
Thus, whenever $|\sigma^{(1)}_a|<|\sigma^{(2)}_a|$, we have $|\sigma^{(1)}_b|\leq|\sigma^{(2)}_b|$. Hence, we may assume that
\[
\sigma^{(1)}_a\prec\sigma^{(2)}_a\;\;{\rm and}\;\;\sigma^{(1)}_b\prec\sigma^{(2)}_b.
\]
In this case we have
\[
([\sigma^{(1)}_a]\times[\sigma^{(1)}_b])\supset([\sigma^{(2)}_a]\times[\sigma^{(2)}_b]).
\]
Let $H_{1,r}$ be as defined in (\ref{s13}). Then by the proof of Lemma \ref{pre0}, we have $|\sigma^{(2)}|\leq|\sigma^{(1)}|+H_{1,r}$. For every $\sigma\in\Upsilon_{j,r}$, we write
\[
\mathcal{F}(\sigma):=\{\omega\in\Upsilon_{j,r}:\sigma_a,\omega_a; \;{\rm and}\; \sigma_b,\omega_b\;{\rm are\;both\;comparable}\}.
\]
Let $\widetilde{\sigma}$ denote the shortest word in $\mathcal{F}(\sigma)$ and $\mathcal{F}_{j,r}$ the set of all such words. Then, For every pair $\widetilde{\sigma},\widetilde{\omega}\in\mathcal{F}_{j,r}$, we have
\begin{equation}\label{s5}
([\widetilde{\sigma}_a]\times[\widetilde{\sigma}_b])\cap([\widetilde{\omega}_a]\times[\widetilde{\omega}_b])=\emptyset.
\end{equation}
By Lemma \ref{pre0}, we have
\begin{equation}\label{s4}
\sum_{\omega\in\mathcal{F}(\sigma)}W([\omega_a]\times[\omega_b])=\sum_{\omega\in S(\widetilde{\sigma})}W([\omega_a]\times[\omega_b])\leq H_{1,r} W([\widetilde{\sigma}_a]\times[\widetilde{\sigma}_b])
\end{equation}
Combining this with (\ref{s3})-(\ref{s5}), we deduce
\begin{eqnarray}\label{g4}
\sum_{\sigma\in\Upsilon_{j,r}}\mathcal{E}_r(\sigma)&\leq&\sum_{\sigma\in\Upsilon_{j,r}}W([\sigma_a]\times[\sigma_b])
\nonumber\\&=&\sum_{\widetilde{\sigma}\in\mathcal{F}_{j,r}}\sum_{\sigma\in\mathcal{F}(\widetilde{\sigma})}W([\sigma_a]\times[\sigma_b])\nonumber\\&\leq&
 H_{1,r}\sum_{\widetilde{\sigma}\in\mathcal{F}_{j,r}}W([\widetilde{\sigma}_a]\times[\widetilde{\sigma}_b])\nonumber\\&\leq& H_{1,r}.
\end{eqnarray}
This, together with (\ref{newanti}), implies
\begin{equation}
\psi_{j,r}\underline{\eta}_r^{\frac{(j+1)s_r}{s_r+r}}\leq H_{1,r},\;\;{\rm implying}\;\;\underline{\eta}_r^{\frac{jr}{s_r+r}}\leq H_{1,r}^{\frac{r}{s_r}}\underline{\eta}_r^{\frac{-r}{s_r+r}}\psi_{j,r}^{-\frac{r}{s_r}}.
\end{equation}
Using this, (\ref{characterization}) and (\ref{g4}), we have
\begin{eqnarray}\label{z1}
e_{\psi_{j,r},r}^r(\mu)&\asymp&\sum_{\sigma\in\Upsilon_{j,r}}\mu_\sigma m^{-|\sigma|r}=\sum_{\sigma\in\Upsilon_{j,r}}\mathcal{E}_r(\sigma)(\mu_\sigma m^{-|\sigma|r})^{\frac{r}{s_r+r}}\nonumber\\&\leq&\sum_{\sigma\in\Upsilon_{j,r}}\mathcal{E}_r(\sigma)\underline{\eta}_r^{\frac{jr}{s_r+r}}\leq H_{1,r}^{1+\frac{r}{s_r}}\underline{\eta}_r^{\frac{-r}{s_r+r}}\psi_{j,r}^{-\frac{r}{s_r}}.
\end{eqnarray}
By Lemma 1 in \cite{KZ:15}, we have $\psi_{j,r}\leq\psi_{j+1,r}\leq (mn)^{H_{1,r}}\psi_{j,r}$. For each $k\geq \psi_{1,r}$, there exists some $j$ such that $\psi_{j,r}\leq k<\psi_{j+1,r}$. Thus, by (\ref{z1}) and  Theorem 4.12 of \cite{GL:00}, we deduce
\begin{eqnarray}\label{g7}
\overline{Q}_r^{s_r}(\mu)&=&\limsup_{k\to\infty}k^{\frac{r}{s_r}}e_{k,r}^r(\mu)\leq\limsup_{j\to\infty}\psi_{j+1,r}^{\frac{r}{s_r}}
e_{\psi_{j,r},r}^r(\mu)\nonumber\\&\leq& (mn)^{\frac{rH_{1,r}}{s_r}}\limsup_{j\to\infty}\psi_{j,r}^{\frac{r}{s_r}}
e_{\psi_{j,r},r}^r(\mu)\nonumber\\&\leq&(mn)^{\frac{rH_{1,r}}{s_r}}H_{1,r}^{1+\frac{r}{s_r}}\underline{\eta}_r^{\frac{-r}{s_r+r}}.
\end{eqnarray}
The proof of the proposition is now complete.
\end{proof}

\section{The positivity of the lower quantization coefficient for $\mu$}

Let $\Upsilon_{j,r}$ be as defined in (\ref{newanti}). We write
\begin{eqnarray*}
k_{1j}:=\min_{\sigma\in\Upsilon_{j,r}}|\sigma|,\;k_{2j}:=\max_{\sigma\in\Upsilon_{j,r}}|\sigma|;\;\Lambda_{j,r}(k):=\Upsilon_{j,r}\cap\Omega_k.
\end{eqnarray*}
For $\sigma\in G^*$ and $\omega\in G_y^*$, we write $\sigma\times\omega$ for the corresponding word in $G^*\times G_y^*$.
We consider words of $G^*\times G_y^*$ which takes the following form:
\[
\sigma\times\omega,\;|\sigma|+\ell(k_{1j})=\ell(|\sigma|+|\omega|+k_{1j}),\;\sigma\in G^*,\omega\in G_y^*.
\]
Let $\mathcal{H}_{j,r}$ denote the set of all such words. Note that
\begin{eqnarray*}
\ell(|\sigma|+|\omega|+k_{1j}-1)&=&[(|\sigma|+|\omega|+k_{1j}-1)\theta]\geq(|\sigma|+|\omega|+k_{1j}-1)\theta-1\\
&=&(|\sigma|+|\omega|+k_{1j})\theta-(1+\theta)>|\sigma|+\ell(k_{1j})-2.
\end{eqnarray*}
Thus, $\ell(|\sigma|+|\omega|+k_{1j}-1)$ takes two possible values: $|\sigma|+\ell(k_{1j})$, or, $|\sigma|+\ell(k_{1j})-1$.
This allows us to define  $(\sigma\times\omega)^\flat\in\mathcal{H}_{j,r}$:
\begin{eqnarray}\label{s7}
(\sigma\times\omega)^\flat:=\left\{ \begin{array}{ll}
\sigma\times\omega^-\;\;{\rm if}\;\;\ell(|\sigma|+|\omega|+k_{1j}-1)=|\sigma|+\ell(k_{1j})\\
\sigma^-\times\omega\;\;{\rm if}\;\;\ell(|\sigma|+|\omega|+k_{1j}-1)=|\sigma|+\ell(k_{1j})-1
\end{array}\right.,
\end{eqnarray}
where $\sigma^-,\omega^-$ are as defined in section 3.
We write $P(\sigma\times\omega):=[\sigma]\times[\omega]$ and
$P((\sigma\times\omega)^\flat):=[\sigma]\times[\omega^-]\;{\rm or}\;[\sigma^-]\times[\omega]$ in accordance with (\ref{s7}). One can easily see
\begin{eqnarray}\label{s8}
P_r^{-1}\underline{\eta}_r^{\frac{s_r}{s_r+r}} W(P(\sigma\times\omega)^\flat)\leq W(P(\sigma\times\omega))<W(P(\sigma\times\omega)^\flat).
\end{eqnarray}

By the definition, for two words $\sigma^{(i)}\times\omega^{(i)}\in\mathcal{H}_{j,r},i=1,2$, if $|\sigma^{(1)}|<|\sigma^{(2)}|$, we have $|\omega^{(1)}|\leq|\omega^{(2)}|$. Thus, whenever $\sigma^{(1)}\prec\sigma^{(2)}$ and $\sigma^{(1)}\neq\sigma^{(2)}$, we have $\omega^{(1)}\prec\omega^{(2)}$.

We write $\sigma^{(1)}\times\omega^{(1)}\prec\sigma^{(2)}\times\omega^{(2)}$, if $\sigma^{(1)}\prec\sigma^{(2)}$ and $\omega^{(1)}\prec\omega^{(2)}$; if neither $\sigma^{(1)}\times\omega^{(1)}\prec\sigma^{(2)}\times\omega^{(2)}$, nor $\sigma^{(1)}\times\omega^{(1)}\prec\sigma^{(2)}\times\omega^{(2)}$, then we say that $\sigma^{(i)}\times\omega^{(i)}\in\mathcal{H}_{j,r},i=1,2$ are incomparable. A finite set $\Gamma\subset\mathcal{H}_{j,r}$ is called a finite maximal antichain, if the words in $\Gamma$ are pairwise incomparable, and for every word $\sigma\times\omega$ in $G^{\mathbb{N}}\times G_y^{\mathbb{N}}$, there exists some word $\sigma'\times\omega'$ such that $\sigma'\prec\sigma$ and $\omega'\prec\omega$; for such a $\Gamma$ in $\mathcal{H}_{j,r}$, we have
\begin{eqnarray}\label{s15}
\bigcup_{\sigma\times\omega\in\Gamma}[\sigma]\times[\omega]=G^{\mathbb{N}}\times G_y^{\mathbb{N}};
\end{eqnarray}
and for every pair of distinct words $\sigma^{(1)}\times\omega^{(1)},\sigma^{(2)}\times\omega^{(2)}\in\Gamma$, we have
\[
([\sigma^{(1)}]\times[\omega^{(1)}])\cap([\sigma^{(2)}]\times[\omega^{(2)}])=\emptyset.
\]

In order to establish a lower bound for the lower quantization coefficient for $\mu$, we will construct a family of subsets of $G^{\mathbb{N}}\times G_y^{\mathbb{N}}$ and associate them with approximate squares. The following lemma will be used to estimate the possible overlapping cases in this process. Recall that for $\sigma,\omega\in\Omega^*$, $\sigma\prec\omega$ means $F_\omega\subset F_\sigma$.
\begin{lemma}\label{pre2}
Let $\sigma\in\Omega^*$ and $H_{2,r}:=P_r^3Q_r^{-2}\underline{\eta}_r^{-\frac{s_r}{s_r+r}}$. We write
\begin{equation}\label{s6}
S_2(\sigma):=\{\omega\in\Omega^*:\;\sigma\prec\omega,\mathcal{E}_r(\omega)\geq H_{2,r}^{-1}\mathcal{E}_r(\sigma)\}.
\end{equation}
Then there exists a constant $H_{3,r}$, which is independent of $\sigma$, such that
\[
\sum_{\omega\in S_2(\sigma)}\mathcal{E}_r(\omega)\leq H_{3,r}\mathcal{E}_r(\sigma).
\]
\end{lemma}
\begin{proof}
Let $\overline{\eta}_r$ be as defined in Section 3.
Write
\begin{eqnarray*}
&\Lambda(\sigma,h):=\{\omega\in\Omega^*:|\omega|=|\sigma|+h,\sigma\prec\omega\},\;h\geq 1;\\
&M_r:=\min\{h\in\mathbb{N}:\overline{\eta}_r^{\frac{hs_r}{s_r+r}}<H_{2,r}^{-1}\}.
\end{eqnarray*}
Then for every $\omega\in\Lambda(\sigma,M_r)$, by (\ref{g5}), we have
\[
\mathcal{E}_r(\omega)\leq\overline{\eta}_r^{\frac{M_rs_r}{s_r+r}}\mathcal{E}_r(\sigma)<H_{2,r}^{-1}\mathcal{E}_r(\sigma).
\]
Hence, for every $\omega\in S_2(\sigma)$, we have $|\omega|\leq|\sigma|+M_r$. It follows that
\[
S_2(\sigma)\subset\bigcup_{h=0}^M\Lambda(\sigma,h).
\]
Note that $0<Q_r\leq1$. By (\ref{s14}), we also have
\begin{eqnarray*}
\sum_{\omega\in\Lambda(\tau,1)}\mathcal{E}_r(\omega)&\leq& Q_r\sum_{i\in G_{x,j_{\ell(k)+1}}}\bigg(\frac{p_{ij_{\ell(k)+1}}}{q_{j_{\ell(k)+1}}}\bigg)^{\frac{s_r}{s_r+r}}\mathcal{E}_r(\tau)
\\&\leq&\max_{j\in G_y}\sum_{i\in G_{x,j}}\bigg(\frac{p_{ij}}{q_j}\bigg)^{\frac{s_r}{s_r+r}}\mathcal{E}_r(\tau)=:\xi_r\mathcal{E}_r(\tau).
\end{eqnarray*}
Using this fact and finite induction, we further deduce
\begin{eqnarray*}
\sum_{\omega\in S_2(\sigma)}\mathcal{E}_r(\omega)\leq \sum_{h=0}^{M_r}\sum_{\omega\in \Lambda(\tau,h)}\mathcal{E}_r(\omega)\leq \sum_{h=0}^{M_r}\xi_r^h\mathcal{E}_r(\sigma).
\end{eqnarray*}
Setting $H_{3,r}:=\sum_{h=0}^{M_r}\xi_r^h$, the lemma follows.
\end{proof}

Using Lemma \ref{pre1} and the product measure $W$, we are now able to prove the positivity of the lower quantization coefficient for $\mu$.
\begin{proposition}\label{pre3}
Let $\mu$ be a measure as defined in (\ref{selfaffinemeas}). Then $\overline{Q}_r^{s_r}(\mu)>0$.
\end{proposition}
\begin{proof}
For every $\tau\in\Omega_{k_{1j}}\setminus\Lambda_{j,r}(k_{1j})$, by (\ref{newanti}), we have $\mathcal{E}_r(\tau)\geq\underline{\eta}_r^{\frac{js_r}{s_r+r}}$. Set
\[
\epsilon(\tau):=\underline{\eta}_r^{\frac{js_r}{s_r+r}}\mathcal{E}_r(\tau)^{-1}.
\]
Then clearly $\epsilon(\tau)\leq 1$ for all $\tau\in\Omega_{k_{1j}}\setminus\Lambda_{j,r}(k_{1j})$. We define
\begin{eqnarray*}
\Gamma(\tau):=\{\sigma\times\omega\in\mathcal{H}_{j,r}:W(P((\sigma\times\omega)^\flat)\geq\epsilon(\tau)>W(P(\sigma\times\omega)\}.
\end{eqnarray*}
Then $\Gamma(\tau)$ is a finite maximal antichain in $\mathcal{H}_{j,r}$. Using (\ref{s15}), we deduce
\begin{eqnarray*}
\sum_{\sigma\times\omega\in\Gamma(\tau)}W([\tau_a\ast\sigma]\times[\tau_b\ast\omega])&=&\sum_{\sigma\times\omega\in\Gamma(\tau)}
\lambda([\tau_a\ast\sigma])\nu([\tau_b\ast\omega])\\&=&\sum_{\sigma\times\omega\in\Gamma(\tau)}
\lambda([\tau_a])\lambda([\sigma])\nu([\tau_b])\nu([\omega])\\&=&W([\tau_a\ast\tau_b])\sum_{\sigma\times\omega\in\Gamma(\tau)}
W([\sigma]\times[\omega])\\&=&W([\tau_a\ast\tau_b]).
\end{eqnarray*}
We need to note the following facts:
\begin{enumerate}
\item[\rm (A)] For every $\tau\in\Omega_{k_{1j}}\setminus\Lambda_{j,r}(k_{1j})$ and $\sigma\times\omega\in\Gamma(\tau)$, by (\ref{g3}), $\tau_a\ast\sigma\ast\tau_b\ast\omega$ is a location code for an approximate square;
\item[\rm (B)]
For distinct $\sigma^{(i)}\times\omega^{(i)}\in\Gamma(\tau),i=1,2$, we have, either $\sigma^{(1)},\sigma^{(2)}$; or $\omega^{(1)},\omega^{(2)}$ are incomparable. Hence,
\begin{eqnarray*}
(\tau_a\ast\sigma^{(1)})\ast(\tau_b\ast\omega^{(1)})\neq(\tau_a\ast\sigma^{(2)})\ast(\tau_b\ast\omega^{(2)});\\
\big[\tau_a\ast\sigma^{(1)}\big]\times\big[\tau_b\ast\omega^{(1)}\big]\neq\big[\tau_a\ast\sigma^{(2)}]\times[\tau_b\ast\omega^{(2)}\big].
\end{eqnarray*}
\item[\rm (C)]For different $\tau^{(i)}\in\Omega_{k_{1j}}\setminus\Lambda_{j,r}(k_{1j}),i=1,2$, we have
\[
|\tau^{(1)}_a|=|\tau^{(2)}_a|,\;\;|\tau^{(1)}_b|=|\tau^{(2)}_b|.
\]
Since $\tau^{(1)}\neq \tau^{(2)}$, we have either $\tau^{(1)}_a,\tau^{(2)}_a$ are incomparable, or $\tau^{(1)}_b,\tau^{(2)}_b$ are incomparable. Thus, for every pair $\sigma^{(i)}\times\omega^{(i)}\in\Gamma(\tau_i),i=1,2$, we have
\begin{eqnarray*}
(\tau^{(1)}_a\ast\sigma^{(1)})\ast(\tau^{(1)}_b\ast\omega^{(1)})\neq(\tau^{(2)}_a\ast\sigma^{(2)})\ast(\tau^{(2)}_b\ast\omega^{(2)});\\
\big[\tau^{(1)}_a\ast\sigma^{(1)}\big]\times\big[\tau^{(1)}_b\ast\omega^{(1)}\big]\neq\big[\tau^{(2)}_a\ast\sigma^{(2)}]\times[\tau^{(2)}_b\ast\omega^{(2)}\big].
\end{eqnarray*}
\item[\rm (D)]
It may happen that
\[
F_{\tau_a\ast\sigma^{(1)})\ast(\tau_b\ast\omega^{(1)})}\subset F_{\tau_a\ast\sigma^{(2)})\ast(\tau_b\ast\omega^{(2)})}.
\]
\end{enumerate}

We denote by $\mathcal{L}_{j,r}(1)$ the set of all the words $(\tau_a\ast\sigma)\ast(\tau_b\ast\omega)$ and words in $\Lambda_{j,r}(k_{1j})$:
\[
\mathcal{L}_{j,r}(1):=\Lambda_{j,r}(k_{1j})\cup\bigg(\bigcup_{\tau\in\Omega_{k_{1j}}\setminus\Lambda_{j,r}(k_{1j})}
\big\{(\tau_a\ast\sigma)\ast(\tau_b\ast\omega):\sigma\times\omega\in\Gamma(\tau)\big\}\bigg).
\]
For every $\tau\in\Omega_{k_{1j}}\setminus\Lambda_{j,r}(k_{1j})$ and $\sigma\times\omega\in\Gamma(\tau)$, using (\ref{s3}) and (\ref{s8}), we have
\begin{eqnarray*}
\mathcal{E}_r((\tau_a\ast\sigma)\ast(\tau_b\ast\omega))&=&(\mu_{\tau_a\ast\sigma)\ast(\tau_b\ast\omega)}m^{-|\tau_a\ast\sigma)\ast(\tau_b\ast\omega)|r})^{\frac{s_r}{s_r+r}}
\\&\geq& P_r^{-1}Q_r W([\tau_a\ast\sigma]\times[\tau_b\ast\omega])\\&=&P_r^{-1}Q_r W([\tau_a]\times[\tau_b])W([\sigma]\times[\omega])
\\&=&P_r^{-1}Q_r W([\tau_a]\times[\tau_b])W(P(\sigma\times\omega))\\&\geq& P_r^{-1}Q_r\mathcal{E}_r(\tau)W(P(\sigma\times\omega))
\\&\geq&P_r^{-1}Q_r\mathcal{E}_r(\tau)P_r^{-1}\underline{\eta}_r^{\frac{s_r}{s_r+r}}W(P((\sigma\times\omega)^\flat))
\\&\geq& P_r^{-1}Q_r\mathcal{E}_r(\tau) P_r^{-1}\underline{\eta}_r^{\frac{s_r}{s_r+r}}
\underline{\eta}_r^{\frac{js_r}{s_r+r}}\mathcal{E}_r(\tau)^{-1}\\&=&P_r^{-2}Q_r\underline{\eta}_r^{\frac{(j+1)s_r}{s_r+r}}.
\end{eqnarray*}
Analogously, one can see that $\mathcal{E}_r((\tau_a\ast\sigma)\ast(\tau_b\ast\omega))\leq P_rQ_r^{-1}\underline{\eta}_r^{\frac{js_r}{s_r+r}}$. In addition, for every $\tau\in\Lambda_{j,r}(k_{1j})$, by (\ref{newanti}), one can see that
\[
\underline{\eta}_r^{\frac{(j+1)s_r}{s_r+r}}\leq\mathcal{E}_r(\tau)<\underline{\eta}_r^{\frac{js_r}{s_r+r}}.
\]
Thus, for all words $\rho\in\mathcal{L}_{j,r}(1)$, we have
\begin{eqnarray}\label{s9}
P_r^{-2}Q_r\underline{\eta}_r^{\frac{(j+1)s_r}{s_r+r}}\leq\mathcal{E}_r(\rho)<P_r Q_r^{-1}\underline{\eta}_r^{\frac{js_r}{s_r+r}}.
\end{eqnarray}
For every $\rho\in\mathcal{L}_{j,r}(1)$, we write
\[
T(\rho):=\{\omega\in\mathcal{L}_{j,r}(1):F_\omega\subset F_\rho,\;{\rm or}\;F_\rho\subset F_\omega\}.
\]
We choose the shortest word of $T(\sigma)$ and denote $\mathcal{L}_{j,r}(2)$ the set of all such words. Then $F_\sigma,\sigma\in\mathcal{L}_{j,r}(2)$ are pairwise non-overlapping. By Lemma \ref{pre2} and (\ref{s3}),
\begin{eqnarray}\label{s10}
\sum_{\rho\in\mathcal{L}_{j,r}(2)}\mathcal{E}_r(\rho)&\geq& H_{3,r}^{-1}\sum_{\rho\in\mathcal{L}_{j,r}(2)}\sum_{\omega\in T(\rho)}\mathcal{E}_r(\omega)=H_{3,r}^{-1}\sum_{\rho\in\mathcal{L}_{j,r}(1)}\mathcal{E}_r(\rho)\nonumber
\\&\geq&H_{3,r}^{-1} P_r^{-1} Q_r\sum_{\tau\in\Lambda_{j,r}(k_{1j})}W([\tau_a]\times[\tau_b])+\nonumber\\&&+H_{3,r}^{-1}P_r^{-1} Q_r
\sum_{\tau\in\Omega_{k_{1j}}\setminus\Lambda_{j,r}(k_{1j})}\sum_{\sigma\times\omega\in\Gamma(\tau)}W([\tau_a\ast\sigma]\times[\tau_b\ast\omega])
\nonumber\\&\geq&H_{3,r}^{-1}P_r^{-1} Q_r\sum_{\tau\in\Lambda_{j,r}(k_{1j})}W([\tau_a]\times[\tau_b])\nonumber\\&&+H_{3,r}^{-1}P_r^{-1} Q_r
\sum_{\tau\in\Omega_{k_{1j}}\setminus\Lambda_{j,r}(k_{1j})}W([\tau_a]\times[\tau_b])
\nonumber\\&=&H_{3,r}^{-1}P_r^{-1} Q_r\sum_{\tau\in\Omega_{k_{1j}}}W([\tau_a]\times[\tau_b])=H_{3,r}^{-1}P_r^{-1} Q_r.
\end{eqnarray}
Analogously, one may show that
\begin{eqnarray}\label{s11}
\sum_{\rho\in\mathcal{L}_{j,r}(2)}\mathcal{E}_r(\rho)\leq 1.
\end{eqnarray}
We denote by $\phi_{j,r}$ the cardinality of $\mathcal{L}_{j,r}(2)$. Then by (\ref{s9})-(\ref{s11}), we deduce
\[
\phi_{j,r}P_r^{-2}Q_r\underline{\eta}_r^{\frac{s_r(j+1)}{s_r+r}}\leq 1;\;\phi_{j,r}P_r Q_r^{-1}\underline{\eta}_r^{\frac{js_r}{s_r+r}}\geq H_{3,r}^{-1}P_r^{-1} Q_r.
\]
Set $H_{4,r}:=P_r^2Q_r^{-1}$ and $H_{5,r}:=H_{3,r}^{-1}P_r^{-2}Q_r^2$. It follows that
\begin{eqnarray}\label{s12}
H_{5,r}\underline{\eta}_r^{\frac{-js_r}{s_r+r}}\leq \phi_{j,r}\leq H_{4,r}\underline{\eta}_r^{\frac{-s_r(j+1)}{s_r+r}}
\end{eqnarray}

Now let $H:=2([\theta^{-1}]+2)$ and $\delta:=(n^2+1)^{-1/2}$. Using the method in the proof of Lemma 2 of \cite{KZ:15}, we may choose a $\widetilde{\rho}$ for every word $\rho\in\mathcal{L}_{j,r}(2)$ such that
\[
\rho\prec\widetilde{\rho},\;\;|\widetilde{\rho}|\leq|\rho|+H
\]
and for every pair of distinct words $\rho,\omega$ of $\mathcal{L}_{j,r}(2)$, we have
\begin{eqnarray*}
d(F_{\widetilde{\rho}},F_{\widetilde{\omega}})\geq\delta\max\{|F_{\widetilde{\rho}}|,|F_{\widetilde{\omega}}|\}.
\end{eqnarray*}
Let $\alpha\in C_{\phi_{j,r},r}(\mu)$. Then by Lemma 3 of \cite{KZ:15}, we can find a constant $D$, which is independent of $j$, such that
 \begin{eqnarray}\label{g6}
e_{\phi_{j,r},r}^r(\mu)&\geq&\sum_{\rho\in\mathcal{L}_{j,r}(2)}\int_{F_\rho} d(x,\alpha)^rd\mu(x)\geq
\sum_{\rho\in\mathcal{L}_{j,r}(2)}\int_{F_{\widetilde{\rho}}} d(x,\alpha)^rd\mu(x)
\nonumber\\&\geq&D\sum_{\rho\in\mathcal{L}_{j,r}(2)}\mu_{{\widetilde{\rho}}}m^{-|\widetilde{\rho}|r}
\geq\widetilde{D}\sum_{\rho\in\mathcal{L}_{j,r}(2)}\mu_\rho m^{-|\rho|r},
\end{eqnarray}
where $\widetilde{D}:=D\underline{\eta}_r^H$.
Thus, by (\ref{s10}), (\ref{g6}) and H\"{o}lder's inequality with exponent less than one, we further deduce
\begin{eqnarray*}
e_{\phi_{j,r},r}^r(\mu)&\geq&\widetilde{D}\bigg(\sum_{\rho\in\mathcal{L}_{j,r}(2)}(\mu_\rho m^{-|\rho|r})^{\frac{s_r}{s_r+r}}\bigg)^{\frac{s_r+r}{s_r}}\phi_{j,r}^{-\frac{r}{s_r}}
\\&=&\widetilde{D}\bigg(\sum_{\rho\in\mathcal{L}_{j,r}(2)}\mathcal{E}_r(\rho)\bigg)^{\frac{s_r+r}{s_r}}\phi_{j,r}^{-\frac{r}{s_r}}
\\&\geq&\widetilde{D} (H_{3,r}P_r Q_r^{-1})^{-\frac{s_r+r}{s_r}}\phi_{j,r}^{-\frac{r}{s_r}}.
\end{eqnarray*}
By (\ref{s12}), we may choose a smallest integer $H_{6,r}$ such that for every $j$, we have
\[
\phi_{j+H_{6,r},r}\geq H_{5,r}\underline{\eta}_r^{\frac{-(j+H_{6,r})s_r}{s_r+r}}>H_{4,r}\underline{\eta}_r^{\frac{-s_r(j+1)}{s_r+r}}\geq\phi_{j,r}
\]
For this integer $H_{6,r}$ and $j\geq 1$, we also have
\begin{eqnarray*}
\phi_{j+H_{6,r},r}&\leq& H_{4,r}\underline{\eta}_r^{\frac{-(j+H_{6,r}+1)s_r}{s_r+r}}= H_{4,r}\underline{\eta}_r^{\frac{-(H_{6,r}+1)s_r}{s_r+r}}\underline{\eta}_r^{\frac{-js_r}{s_r+r}}\\&\leq& H_{5,r}^{-1}H_{4,r}\underline{\eta}_r^{\frac{-(H_{6,r}+1)s_r}{s_r+r}}\phi_{j,r}.
\end{eqnarray*}
We set $N_{j,r}:=\phi_{[\theta^{-1}+jH_{6,r}],r}$ and $H_{7,r}:=H_{5,r}^{-1}H_{4,r}\underline{\eta}_r^{\frac{-(H_{6,r}+1)s_r}{s_r+r}}$. Then we have
\[
N_{j,r}<N_{j+1,r}\leq H_{7,r} N_{j,r},\;\;N_{j,r}^{\frac{r}{s_r}}e_{N_{j,r}}^r(\mu)\geq\widetilde{D} (H_{3,r}P_r Q_r^{-1})^{-\frac{s_r+r}{s_r}}.
\]
For each $k\geq\phi_1$, we choose $j$ such that $k\in[ N_{j,r}, N_{j+1,r})$. Then using Theorem 4.12 of \cite{GL:00}, we deduce
\begin{eqnarray*}
\underline{Q}_r^{s_r}(\mu)&=&\liminf_{k\to\infty}k^{\frac{r}{s_r}}e_{k,r}^r(\mu)\geq\liminf_{j\to\infty} N_{j,r}^{\frac{r}{s_r}}
e_{ N_{j+1,r},r}^r(\mu)\\&\geq& (H_{7,r})^{-\frac{r}{s_r}}\liminf_{j\to\infty} N_{j+1,r}^{\frac{r}{s_r}}
e_{ N_{j+1,r},r}^r(\mu)\\&\geq&(H_{7,r})^{-\frac{r}{s_r}}\widetilde{D} (H_{3,r}P_r Q_r^{-1})^{-\frac{s_r+r}{s_r}}.
\end{eqnarray*}
This completes the proof of the proposition.
\end{proof}

\emph{Proof of Theorem \ref{mthm1}}
It is an immediate consequence of Proposition \ref{pre1} and \ref{pre3}.


\section*{References}
\begin{thebibliography}{10}


\bibitem{Bed:84} Bedford T 1984 Crinkly curves, Markov partitions
and box dimensions in self-similar sets PhD Thesis, University of
Warwick

\bibitem{Fal:10}Falconer K J 2010 Generalized dimensions of measures
on almost self-affine sets \emph{Nonlinearity} \textbf{23} 1047--69

\bibitem{GL:00} Graf S and Luschgy H 2000 Foundations of quantization
for probability distributions \emph{Lecture Notes in Math.} vol.
1730, Springer

\bibitem{GL:04}Graf S and Luschgy H 2004 Quantization for probabilitiy
measures with respect to the geometric mean error. \emph{Math. Proc.
Camb. Phil. Soc.} \textbf{136} 687--717

\bibitem{GL:05}Graf S and Luschgy H 2005 The point density measure
in the quantization of self-similar probabilities \emph{Math. Proc.
Camb. Phil. Soc.} \textbf{138} 513--31

\bibitem{GN:98} Gray R and Neuhoff D 1998 Quantization \emph{IEEE
Trans. Inform. Theory} \textbf{44} 2325--83

\bibitem{Hut:81} Hutchinson J E 1981 Fractals and self-similarity \emph{
Indiana Univ. Math. J.} \textbf{30} 713--47

\bibitem{J:11} Jordan T and Rams M 2011 Multifractal analysis for
Bedford-McMullen carpets \emph{Math. Proc. Camb. Phil. Soc.} \textbf{150},
147--56

\bibitem{KZ:15} Kesseb\"{o}hmer M and Zhu S 2015 On the quantization for self-affine measures on Bedford-McMullen carpets \emph{Math. Z.}, in press. DOI: 10.1007/s00209-015-1588-3

\bibitem{KZ:15b} Kesseb\"{o}hmer M and Zhu S 2015  Some recent developments in the quantization for probability measures Bandt, C.,
Falconer, K., Z\"{a}le, M. (eds.) Fractal Geometry and Stochastics V, Progress in Probability, vol. 70, pp. 105-120 (2015). Springer, Switzerland

\bibitem{King:95} King J F 1995 The singularity spectrum for general
Sierpi\'{n}ski carpets \emph{Adv. Math.} \textbf{116} 1--11

\bibitem{Kr:08} Kreitmeier W 2008 Optimal quantization for dyadic
homogeneous Cantor distributions \emph{Math. Nachr.} \textbf{281}
1307--27

\bibitem{LG:92} Lalley S P and Gatzouras D 1992 Hausdorff and box
dimensions of certain self-affine fractals. \emph{Indiana Univ. Math.
J.} \textbf{41} 533--68

\bibitem{LM:02}Lindsay L J and Mauldin R D 2002 Quantization dimension
for conformal iterated function systems. \emph{Nonlinearity} \textbf{15}
189--99

\bibitem{Mcmullen:84} McMullen C 1984 The Hausdorff dimension of
general Sierpi\'{n}ski carpetes. \emph{Nagoya Math. J.} \textbf{96}
1--9

\bibitem{Peres:94b} Peres Y 1994 The self-affine carpetes of McMullen
and Bedford have infinite Hausdorff measure, \emph{Math. Proc. Camb.
Phil. Soc.} \textbf{116} 513--26

\bibitem{PK:01} P\"{o}tzelberger K 2001 The quantization dimension of
distributions. \emph{Math. Proc. Camb. Phil. Soc.} \textbf{131} 507--19

\bibitem{Za:63} Zador P L 1964 Development and evaluation of procedures
for quantizing multivariate distributions, PhD. Thesis, Stanford
University

\end{thebibliography}
\end{document}